\documentclass[10pt]{amsart}
\usepackage{graphicx}
 \usepackage[parfill]{parskip}
 \usepackage{amsmath}
\usepackage{amsthm}
\usepackage{amssymb}
\usepackage[T1]{fontenc}
 \newtheorem{theorem}{Theorem}
 \newtheorem{conjecture}{Conjecture}
\newtheorem{lemma}[theorem]{Lemma}
\newtheorem{remark}{Remark}
\newtheorem{proposition}[theorem]{Proposition}
\newtheorem{corollary}[theorem]{Corollary}

\def\a{\alpha}

\def\aa{|A|}

\def\E{\mathsf {E}}

\def\eps{\varepsilon}
\def\d{\delta}

\def\F{{\mathbb F}}

\title{On the energy variant of the sum-product conjecture}

\author{Misha Rudnev}
\address{Misha Rudnev, Department of Mathematics, University of Bristol,
  Bristol BS8 1TW, United Kingdom}
\email{m.rudnev@bristol.ac.uk}

\author{Ilya D. Shkredov}
\address{Ilya D. Shkredov, Steklov Mathematical Institute, Division of Number Theory, ul. Gubkina, 8, Moscow, 119991; IITP RAS, Bolshoy Karetny per. 19, Moscow, 127994 and MIPT, Institutskii per. 9, Dolgoprudnii, 141701}
\email{ilya.shkredov@gmail.com}

\author{Sophie Stevens}
\address{Sophie Stevens, Department of Mathematics, University of Bristol,
  Bristol BS8 1TW, United Kingdom}
\email{sophie\_ccs.2011@bristol.ac.uk}

\begin{document}
\begin{abstract} We prove new exponents for the energy version of the Erd\H os-Szemer\'edi sum-product conjecture, raised by Balog and Wooley. They match the previously established milestone values for the standard formulation of the question, both for general fields and the  special case of real or complex numbers, and appear to be the best ones attainable within the currently available technology. Further results are obtained about multiplicative energies of additive shifts and a strengthened energy version of the ``few sums, many products"  inequality of Elekes and Ruzsa. The latter inequality enables us to obtain a minor improvement of the state-of the art sum-product exponent over the reals due to Konyagin and the second author, up to $\frac{4}{3}+\frac{1}{1509}$. An application of energy estimates to an instance of arithmetic growth in prime residue fields is presented.
\end{abstract}

\maketitle

\section{Preface}
In this paper we show that the milestone results in the current sum-product theory literature allow for a pure energy formulation involving both addition and multiplication. Previous inequalities of sum-product type involve either $|A\cdot A|$ and $|A+A|$, or  $\E^+(A)$ and $|A\cdot A|$, or $\E^*(A)$ and $|A+A|$. The energy-energy formulation was raised 
by Balog and Wooley \cite{BW}; its key feature being that the sum-product conjecture cannot hold in its maximum strength in the energy formulation. Owing to an example in \cite{BW}, which we will shortly retell, the cardinality  formulation as Conjecture \ref{esc} below should be weakened to the energy version as Conjecture \ref{esc1}. The latter suggests a somewhat uncomfortably-looking fractional exponent. 

We advance largely on the technical level. The well-established tool in additive combinatorics for passing from energy-type results to the existence of subsets with the desired structure is the Balog-Szemer\'edi-Gowers theorem. This has been done, in particular, by Balog and Wooley, as well as much earlier work on the relation between geometric incidences and algebraic growth, such as \cite{BKT}, \cite{GT}. But for someone concerned with quantitative values of the resulting exponents, the use of the Balog-Szemrer\'edi-Gowers theorem usually comes at a price.

It was first observed that the Balog-Szemer\'edi-Gowers theorem can be avoided by Konyagin and the second author \cite{KS2} who succeeded in significantly strengthening the main estimate of \cite{BW}. But the final exponent in the resulting energy-energy inequality in \cite{KS2} has nonetheless stopped short of its cardinality-cardinality predecessor due to Elekes \cite{E}.

In this paper we remove this gap and prove a variety of energy-energy sum-product type inequalities, which have the same exponents as their cardinality prototypes, up to logarithmic factors. The first of our key inequalities dealing with real/complex numbers is \eqref{formula3}, which is the energy-energy analogue of the classical sum-product $5/4$ result by Elekes \cite{E}, which we record here as \eqref{ele}. The two energies appearing in this inequality can be replaced by energies of two different additive shifts of a set (which is useful for applications), see Theorem \ref{t:BW_mult}. The second  inequality is given by  \eqref{formula4}, which, in a certain sharp regime, is the energy version of the ``few sums, many products" inequality by Elekes and Ruzsa \cite{ER}. As a matter of fact  \eqref{formula4} has better exponents than its prototype, see the discussion before the formulation of the corresponding Theorem \ref{t:fsmp}.

We also furnish the general field variants of our energy inequalities, valid in particular in positive characteristic, with an inevitable constraint on how small the set in question should be. They are usually just slightly weaker than for the particular case of the real or complex field. This combines the pruning techniques developed in this paper with the use of the incidence theorem of the first author from \cite{misha} and its further development in  \cite{RRS}, \cite{AMRS}. 

We strongly feel (although we make no attempt to substantiate this claim), that these energy exponents are the best ones attainable within the currently available technology. For the real or complex field there are better sum-product exponents, based on the foundational work by Solymosi \cite{S1}, \cite{Sol05}. But these appear to necessarily involve cardinality of at least one counterpart, rather than the two energies. For general fields, there are no better ones so far. It appears that further significant progress towards the sum-product conjecture challenges one to break these energy barriers.

Finally, in the main body of the paper we discuss some applications of our energy-energy sum-product type inequalities: some in passing and one at length.

\section{Introduction and main results}

Let $\F$ be a field, with the multiplicative group $\F^*$ and let $A\subset \F$ be a finite set. The sum set, product set and quotient set of $A$ are defined  respectively as
\begin{align*}
&A+A:=\{a+b:a,b\in A\},\\
&A\cdot A:=\{a\cdot b:a,b\in A\},\\
&A/A:=\{a/b ~:~ a,b\in{A},\,b\neq0\}.
\end{align*}
Sum, etc. sets involving different sets, say $A+B$ are defined in a similar way.  For $x\in \F^*$, we write simply $xA$ for $\{x\}A$ and $x+A$ or $A+x$ for $\{x\}+A$.

It has long been observed that unless $A$ is in some sense close to a coset of a subfield of $\F$, additive and multiplicative structure find it hard to coexist : it is easy to minimise one of $|A+A|$ or $ |A\cdot A|$ but then the other set becomes very large. The question was raised --  originally in the context of integers -- and the first quantitative result obtained by Erd\H{o}s and Szemer\'{e}di \cite{ES}, leading to the renowned conjecture:

\begin{conjecture} \label{esc}
Let $A\subset \mathbb{R}$. Then
\begin{equation*}
\max(|A + A|,\, |A \cdot A|)\, \gtrsim\, |A|^{2}.
\end{equation*}
\end{conjecture}

To avoid trivialities we further assume the $0\not\in A$ and $|A|>1$, this will be implicit in all statements we make, as well as that the sets $A,B,C,\dots$ are finite.

As usual, we use the notation $|\cdot|$ for cardinalities of finite sets. The symbols $\ll$, $\gg,$ suppress absolute constants in inequalities, as do their respective equivalents $O$ and $\Omega$. Besides,  $X=\Theta(Y)$ means that $X=O(Y)$ and $X=\Omega(Y)$.

In addition, e.g., in the above statement of Conjecture \ref{esc}, the symbols $\lesssim, \; \gtrsim, \; \sim$ are used to replace, respectively $O,\;\Omega,\;\Theta$ when the inequalities involved are weakened by  a power of $\log|A|$. Thus $|A|$ is viewed as an asymptotic parameter. The suppressed constants  are therefore independent of $|A|$.


If $\F$ has positive characteristic $p$, we only deal with a ``small set'' case, $|A|$ being at most some less-than-$1$ power of $p$, so $p$, always denoting the positive characteristic of $\F$,  is regarded as an asymptotic parameter as well.

The sum-product conjecture remains open although current world records in  \cite{KS2} for the reals and \cite{RRS}, \cite{AMRS} in positive characteristic have edged nearer to the statement of the conjecture.

Fundamental to the study of the sum-product phenomenon is a $L^2$ quantity expressing the additivity or multiplicativity of a set, known as energy. The additive energy between sets $A$ and $B$ is defined as
\begin{equation*}
\E^{+}(A,B):=|\{(a_1,a_2;b_1,b_2)\in A^{2}\times B^2 :a_1+b_1=a_2+b_2\}|.
\end{equation*}
We write $\E^+(A,A)=\E^+(A)$; if this quantity is considerably larger than the trivial lower bound $|A|^2$ (and closer to the trivial upper bound $|A|^3$), we say $A$ has an additive structure.

The multiplicative energy $\E^\times(A,B), \,\E^\times(A)$ is defined similarly.
The Cauchy-Schwarz inequality relates the energy to the Erd\H{o}s-Szemer\'{e}di conjecture: \begin{equation}\label{CSE} \E^{+}(A)|A+A|,~\E^{*}(A)|A\cdot A|\;\geq \;|A|^{4}.\end{equation}
Note that by rearranging the equation defining energy,  say $A+A$ in \eqref{CSE} can be replaced by $A-A$, even though the latter two sets may in principle differ quite a bit.

Geometrically, the multiplicative energy $\E^\times(A)$ equals the number of ordered pairs of points of the plane set $A\times A$, supported on lines through the origin (corresponding to ratios in $A/A$). The additive energy $\E^+(A)$ equals the number of ordered pairs of points of $A\times A$ on parallel lines with the slope plus or minus $1$.

To what quantitative extent do additive and multiplicative structure find it hard to coexist in a set? Recently Balog and Wooley \cite{BW} raised an interesting question as to what is the correct energy, that is $L^2$, formulation of Conjecture \ref{esc}.  If one believes that one of $|A+A|$ or $|A\cdot A|$ must be $\gtrsim|A|^2$ can both $\E^+(A)$ and $\E^\times(A)$ be nonetheless large? An easy example shows, yes: take  $A$  as the union of two disjoint equal in size arithmetic and  geometric progressions. Then both energies are $\Omega(|A|^3)$, that is, up to constants, as big as it gets. Moreover, any subset containing more than, say $51\%$ of  $A$ would have both energies $\Omega(|A|^3)$. On the other hand, for either $\E^+,$ or $\E^\times$, one can find a subset containing every second member of $A$, where the corresponding energy is $\lesssim |A|^2$.

It is possible to intertwine an arithmetic and geometric progression in a smarter way to ensure that any subset of $A$ containing a positive proportion of its members, has both energies considerably in excess of $|A|^2$. Balog and Wooley constructed such an example over the integers. We take a moment to review it briefly, owing to its appeal.

Let $A$ be the union of $n$ disjoint dilates of the integer interval $I= [n^2,\ldots,2n^2)$ by factors $1,2,\ldots, 2^{n-1}$. Note that $A\times A$ contains the union of $n$ disjoint square grids $P=(I\times I) \,\cup\, (2I\times 2I)\,\cup\, \ldots\,\cup \,(2^{n-1}I \times 2^{n-1}I)$. Then $\E^+(A)\gg |A|^{7/3}$ (sum of additive energies of each grid in $P$) while $\E^\times(A)\gg |A|^{7/3}\log^s|A|,$ because of \eqref{CSE} and the fact that $|A\cdot A| \sim n^5$. (See \cite{Fo} proving the explicit value of $s = 0.086...\,$.)  If $A'\subset A$ has cardinality $\alpha|A|$, for $0<\alpha<1$, it is easy to estimate $\E^+(A')$ from below just by looking at  intersections $A'\times A'$ with $P$; the minimum estimate is achieved when  $A'\times A'$ intersects each square forming $P$ uniformly at $(\alpha n^2)^2$ points. For $\E^+(A')$ use \eqref{CSE} and the obvious inclusion $A'\cdot A' \,\subseteq\, A\cdot A.$ Thus
$$\E^+(A'), \,\E^\times(A') \gg \alpha^4 n^7 = \alpha^4 \Theta( |A|^{7/3} ).$$
Besides, $\E^\times(A')$ actually exceeds the right-hand side by a power of $\log|A|$. By slight manipulations with the number $n$ of dilates of $I$ versus its size $n^2$, one can easily ensure that the logarithmic factor in $|A|$ is present in the estimate for both energies of $A'$.

The Balog--Wooley example shows that even though multiplicative and additive structure may not conjecturally coexist in any $A$ in the strong sense of Conjecture \ref{esc}, they cannot be completely divorced in the $L^2$ sense even by taking reasonably small subsets.
On the other hand, we do not believe that the standard construction arsenal offers a stronger one than described above. It seems likely that the following claim is true.

\begin{conjecture} \label{esc1}
Let $A\subset \mathbb{R}$. There exists $A'\subseteq A$, such that $|A'|\geq |A|/2$, and
\begin{equation*}
\min [ \E^+(A'),\,\E^\times(A') ] \lesssim |A|^{7/3}.
\end{equation*}

\end{conjecture}
Note that trivially we cannot expect to destroy multiplicative (respectively additive) structure by taking a subset, unless the latter is very thin, as is the case if, for example, $A$ is a geometric (respectively arithmetic) progression.

Balog and Wooley formulated their results in terms of the decomposition of $A$ as follows.
\begin{theorem}[\cite{BW}]
    Let $A\subset \mathbb R$ be a  set and $\d = 2/33$.
    Then there are two disjoint subsets $B$ and $C$ of $A$ such that $A = B\sqcup C$ and
$$
    \max\{ \E^{+} (B), \E^{\times} (C)\} \lesssim |A|^{3-\delta}.
$$

\label{t:BW}
\end{theorem}

\begin{remark} The sum-product phenomenon is not restricted specifically to reals. Its study in prime residue fields $\F_p$ was initiated by Bourgain, Katz, and Tao \cite{BKT}.  If  the field $\F$ has positive characteristic $p>0$,  we consider the case when $A$ is suitably small in terms of $p$. We do not know of any evidence that Conjectures  \ref{esc}, \ref{esc1} may be false if, say $|A|<p^{1/3}$. To this end \cite{BW} contains a positive characteristic version of Theorem \ref{t:BW}, with a smaller value of
$\delta=4/101,$ subject to the constraint roughly $|A|<p^{101/161}.$\label{rmk}\end{remark}

The two main ingredients of the argument in \cite{BW} were the two following geometric incidence theorems and additive combinatorics. The known sum-product results over the real or complex field are somewhat  stronger than in fields of positive characteristic largely due to order properties of reals, which so far have been indispensable for proofs of  the celebrated Szemer\'edi-Trotter theorem in the plane. See \cite{SzT}, \cite{To} for the original proof for reals and subsequent extension to the complex field.

\begin{theorem} [Szemer\'edi-Trotter Theorem]\label{t:ST} The number of incidences between a set of $m$ lines and $n$ points in $\mathbb C^2$ is $O[(mn)^{2/3} + m + n]$. \end{theorem}
For arbitrary fields the first author proved a weaker geometric incidence theorem in  $\F^3$.
\begin{theorem} [\cite{misha}] \label{t:MR} The number of incidences between a set of $m$ planes and $n\leq m$ points in $\F^3$ is
$O[m (n^{1/2} +k )]$, where $k$ is the maximum number of collinear points, under an additional constraint $n\leq p^2$ in positive characteristic.
\end{theorem}
Thus real or complex numbers will be the special case in the sequel, and since $\mathbb R$ is not special for the ensuing discussion versus $\mathbb C$, we formulate the corresponding results in terms of the latter field.

It was shown in \cite{AMRS} that Theorem \ref{t:MR} implies a weaker version of the Szemer\'edi-Trotter theorem for a general $\F$, with the main term $m^{3/4}n^{2/3}$ if the point set is a Cartesian product $A\times B$, with $|B|\le |A|\leq p^{2/3}$ in positive characteristic. This was recently improved by the third author and de Zeeuw and generalised to arbitrary  point sets as follows.

\begin{theorem} [\cite{SZ}]\label{t:SzT_Fp}
Consider a set of $m$ lines in $\F^2$.

(i) Let $A\times B\subset \F^2$  be a set  of $n$ points, with $n^{1/2}<m<n^{3/2}$ and the constraint $mn^2<p^4$ in positive characteristic. The number of incidences between the above sets of points and lines is  $O(m^{3/4}n^{5/8}).$

(ii) For any set of $n$ points in $\F^2$, with $n^{7/8}<m<n^{8/7}$ and $m^{13}n^{-2} < p^{15}$, the number of incidences with the set of $m$ lines is $O[(mn)^{11/15}].$

\end{theorem}

Incidence theorems have been widely used in arithmetic combinatorics. Elekes \cite{E}  realised that Theorem \ref{t:ST} applies  to sum-product type problems and proved the following estimate towards Conjecture \ref{esc}:
\begin{equation}\label{ele}
\max( |A+A|,\;|A\cdot A| )\gg |A|^{1+\delta},~~\delta=1/4.\end{equation}

Roche-Newton and the first two authors \cite{RRS} applied  Theorem  \ref{t:MR} in a similar vein and proved, for any field $\F$, that
\begin{equation}\label{us}
\max( |A+A|,\;|A\cdot A| )\gg |A|^{1+\delta},~~\delta=1/5.\end{equation}
under an additional constraint $|A|\leq p^{5/8}$ in positive characteristic. 

There has been a series of improvements of the estimate \eqref{ele} in the real and complex case, started by Solymosi \cite{S1}, and currently up to  $\gtrsim |A|^{\frac{4}{3}+\frac{5}{9813}}$ in the right-hand side by Konyagin and the second author \cite{KS2}.  All such improvements of \eqref{ele} (see, e.g., the references in \cite{KS2}) used crucially  the order properties of the reals (the arguments would usually generalise to $\mathbb C$, see for example \cite{KR})  and benefited by repeated applications of the Szemer\'edi-Trotter theorem involving the sets of sums or products themselves. Without order in (a subfield of) $\F$, the sum-product estimate \eqref{us} is the best one known.

\medskip
Our main result is the following theorem, which implies, up to factors of $\log|A|$, the  latter two estimates in the context of Conjecture \ref{esc1}. We establish the following.

\begin{theorem} \label{mainth}
Let $A\subset \F$.There exists $A'\subseteq A$, such that $|A'|\geq  |A|/2$, and
\begin{equation*}
\min [ \E^+(A'),\,\E^\times(A') ] \lesssim |A|^{3-\delta},
\end{equation*}
where $\delta=1/4$ in the special case $\F=\mathbb C$ and $\delta = 1/5$ for any $\F$, with an additional constraint $|A|\leq p^{5/8}$ in positive characteristic.
\end{theorem}

Theorem \ref{mainth} is an immediate consequence of the forthcoming, and stronger, Theorem \ref{t:BW_R}, which  is an improvement of Theorem \ref{t:BW}.  The values of $\delta$ we establish match those in the estimates \eqref{ele}, \eqref{us}. Thus, our arguments emphasise the geometric (and reduce the additive) combinatorics content of the proof
: passing from the estimates \eqref{ele}, \eqref{us} to their weaker  $L^2$ formulation in Theorem \ref{mainth} only incurs logarithmic factors in $|A|$.

Since these estimates are, of course, partial apropos of Conjecture \ref{esc1}, we have not troubled ourselves with calculating the exact powers of $\log|A|$. On the other hand, we do not expect that, modulo these factors, the estimates in question can be improved within today's state of the art toolkit. Once again, all the improvements of  the Elekes estimate \eqref{ele} in the real/complex case appear  to relate (multiplicative or additive) energy to the size of the counterpart  (respectively product or sum) set  and do not work in the  energy-energy  sense, cf. the title of the breakthrough paper \cite{Sol05} by Solymosi. 

These improvements, in particular were enabled by the idea of the second author and Schoen to use the {\em third,} rather the {\em second} moment, or cubic (and higher order) energy of the convolution function arising in the description of sum or product sets. See, e.g., \cite{SS}. This opportunity is inherent in the numerical values of the exponents arising in the Szemer\'edi-Trotter theorem. It does not appear to be granted by the weaker Theorem \ref{t:MR}.
Nor does it seem to be at hand if one pursues energy-energy estimates. This is why, we believe, Theorem \ref{mainth} marks a certain milestone, and to improve its exponents, which are $\delta=1/4$ for $\mathbb F=\mathbb C$ and $\delta =1/5$ otherwise, one needs a conceptual innovation,  whether this is about the reals or a general  $\F$.

On the technical level we do much better than Theorem \ref{t:BW} by avoiding the use of the Balog-Szemer\'{e}di-Gowers theorem. The latter, presented as Theorem \ref{t:BSzG_Schoen} below,  has been a standard arithmetic combinatorial tool for passing from large energy bounds to subsets with small doubling, in particular Balog and Wooley used it to prove Theorem \ref{t:BW}.  Unfortunately, on the quantitative level applying the Balog-Szemer\'{e}di-Gowers theorem is usually quite wasteful. Konyagin and the second author  found a way to avoid it in the context  of Theorem \ref{t:BW}, where they proved $\delta=1/5$ over the reals \cite{KS2}.  In the proof of the key result of this paper, the forthcoming Theorem \ref{t:BW_R}, we follow the main line of the argument in Section 4 of \cite{KS2}, making it somewhat stronger, which yields what we believe is the strongest result,  within reach of today's machinery, that is $\delta=1/4$ for real and complex numbers.

We develop the argument in the context of a general field $\F$, where we use Theorem \ref{t:MR}, while Theorem  \ref{t:ST} applies to $\F=\mathbb C$ as a special case. This enables us to prove $\delta=1/5$ for a general $\F$, matching its value in the sum-product estimate \eqref{us}, while $\delta=1/4$ for $\F=\mathbb C$, matches the Elekes estimate \eqref{ele}. 

As far as applications are concerned, we are interested  in quantitative arithmetic growth estimates.  By arithmetic growth we mean, for an integer $n\geq 2$, having a function $f: \F^n\to \F$,  such that for any $A\subset \F$, sufficiently small in terms of $p$ in the positive characteristic case, the cardinality of the range of $f$, restricted to $A^n$, exceeds $|A|$ by orders of magnitude. See \cite{AMRS} and the references contained therein for some quantitative estimates for $n=2,3,4$ over general fields and  general discussion.

It appears that our energy method enables one to obtain stronger quantitative growth estimates, for they often result in a relation binding energies of two different types, say $\E^\times(A)$ and $\E^+(A)$ for some putative set $A$. Our results, see the forthcoming Corollary \ref{t:BW_C}, provide  upper bounds for the product of the two energies if one passes to a pair of large subsets. Previously available estimates would bind, say multiplicative energy and the sum set, see e.g., \cite{Sol05}. Passing from  $\E^+(A)$ to the sum set, aiming basically to invert the first inequality in \eqref{CSE} would invoke a quantitatively costly application of the  Balog-Szemer\'{e}di-Gowers theorem. Our method enables one to  bypass this. 

There is a well established connection between arithmetic growth and incidence geometry estimates, in both directions. However exploring this connection from the former towards the latter estimates would invariably invoke the Balog-Szemer\'{e}di-Gowers theorem. We limit the references to the well-known works of Bourgain, Katz and Tao \cite{BKT} and Green and Tao \cite{GT}; see Section 6 in both papers. It turns out, however, that  the Balog-Szemer\'{e}di-Gowers theorem can be avoided. In the context of incidence estimates, namely Theorem \ref{t:SzT_Fp} (ii), this was achieved in \cite{SZ}. In the same vein, we challenge an interested reader to embark on reducing -- and strengthening -- the fairly lengthy proof of Proposition 6.6 in \cite{GT} to a much shorter energy argument, avoiding the Balog-Szemer\'{e}di-Gowers theorem  in the vein of the forthcoming Theorem \ref{t:BW_mult}.

We presently limit the number of applications considered in detail to one, concerning the prime residue field $\F_p$. Given $f: \F^n_p\to \F_p$, what is the lower bound on $|A|$, such that for any  $A\subset \F_p$, the range of $f$, restricted to $A^n$, takes up a positive proportion of the field $\F_p$? For many such $f$ one can relatively easily prove the threshold $|A|=\Omega(p^{2/3}),$ via character sums or often just linear algebra methods that work well for relatively large sets with respect to $p$. See e.g \cite{CEIK}. 
However, these techniques usually fail to work for smaller $A$. To this effect, the challenge is to reduce the threshold $|A|=\Omega(p^{2/3})$ for some $f$.

Petridis \cite{Petridis} proved recently that if $|A| \geq p^{5/8}$, then the cardinalities of the sets
$(A+ A)\cdot (A + A), \,(A+A) / (A+A)$ are both $\Omega(p)$, having incorporated the so-called  ``generic projections''  argument  from \cite{BKT} and  Theorem \ref{t:MR} into a rather involved argument. Generic projections is an easy pigeonholing argument showing  that the sets
$(A-A) / (A-A)$, as well as $\{\frac{ac-bd}{a-d}:\,a,b,c,d\in A,\,a\neq d \}$ are both equal to $\F_p$ as long as $|A| >p^{1/2}$.

As a matter of fact, Petridis establishes a stronger $L^2$ claim that the number of solutions of the equation
$$
(a+b)(c+d) = (a'+b')(c'+d'):\,a,\ldots,d'\in A
$$
is bounded as $O(|A|^8/p)$, that is up to a constant the expected number, as long as $|A|>p^{5/8}$. Such a bound appears to be out of reach by methods of \cite{BKT} even regarding the set $(A-A) / (A-A)$ if one rearranges the latter equation as fractions and replaces the plus signs by minuses.

In this paper we  establish the following.
\begin{theorem}\label{thm:had}
Let $A\subseteq \F_p, $ with $|A|  \gg p^{25/42}\log^K |A|$, for some absolute constant $K$. There are disjoint $B,C\subset A$, each of cardinality $\geq |A|/3$, such that  number of solutions to the equation
\begin{equation} \frac{ab -  c}{a  -  d}=\frac{a'b'  - c'}{a' - d'}:\,a,b,a',b' \in B;\,c,d,c',d' \in C  \label{e:had}\end{equation}
is $O({\aa^8}/{p})$, and
therefore
$\left|\left\{
\frac{ab-c}{a-d}:a,b,c,d\in A
\right\}
\right|=\Omega(p)\,.$
\end{theorem}
The reader can verify that in the latter theorem all the minus signs can be replaced by plus signs as well.

\subsection{Further results}
Here we present a somewhat stronger formulation of Theorem \ref{mainth}, its analogue for multiplicative energies of additive shifts and the energy version of the Elekes-Ruzsa few sums, many products inequality. The latter is available only in the real/complex setting; over the reals it yields a minor improvement of the best known sum-product exponent, after being plugged into the argument recently developed by Konyagin and the second author. We also provide an auxiliary subsection which contains a suitably tailored version of the Balog-Szemer\'edi-Gowers theorem -- which may be interesting in its own right -- and some indication of what our results would look like if the Balog-Szemer\'edi-Gowers theorem had to be used.

\subsubsection{Sum-product decomposition and energy inequalities}
\begin{theorem}[Balog-Wooley decomposition]\label{t:BW_R}
Let $A\subset \F$.  There exist two disjoint subsets $B$ and $C$ of $A$, such that   $A=B\sqcup C$, and
\begin{equation*}\max [ \E^+(B),\,\E^\times(C) ] \lesssim |A|^{3-\delta},
\end{equation*}
where $\delta=1/4$ in the special case $\F=\mathbb C$ and $\delta = 1/5$ for any $\F$, with an additional constraint $|A|\leq p^{5/8}$ in positive characteristic.
\end{theorem}

Theorem \ref{t:BW_R} clearly implies Theorem \ref{mainth}: one of $B,C$ has size $\geq|A|/2$.

\medskip
As we have mentioned, one cannot expect both $B$ and $C$ in Theorem \ref{t:BW_R} to constitute a positive proportion of $A$. But this can be achieved by weakening the claim as follows, to be used in the proof of Theorem \ref{thm:had}.
\begin{corollary}\label{t:BW_C}
Let $A\subset \F$, with an additional constraint $|A|\leq p^{3/5}$ in positive characteristic.  There exist two disjoint subsets $B$ and $C$ of $A$, each of cardinality $\geq |A|/3$, such that
\begin{equation}\label{formula1}
 \E^{+} (B)  \cdot  \E^\times(C)^{3/2}  \; \lesssim \;|A|^7.
\end{equation}
In the latter estimate the additive and multiplicative energy can be swapped (for some other $B,C$). Besides, there exist two disjoint subsets $B$ and $C$ of $A$, each of cardinality $\Omega( |A|)$, such that
\begin{equation} \label{formula2} \E^+(B)\cdot \E^\times(C)  \, \lesssim \, |A|^{28/5}.
\end{equation}
Furthermore, if $\mathbb F=\mathbb C$, the estimate \eqref{formula1} improves to 
\begin{equation}\label{formula3}
 \E^{+} (B)  \cdot  \E^\times(C)  \; \lesssim \;|A|^{11/2}.
\end{equation}
\end{corollary}

\bigskip
We will spell out the proof of Theorem \ref{t:BW_R}, the key quantitative result, in all detail. This proof, furthermore, allows for a number of straightforward variations, which result from the fact established in the quoted literature. Some of these variations are left without detailed proofs, for they would  repeat the main arguments more or less line by line.

\subsubsection{Multiplicative energy of translates}

In \cite{Shkredov_R[A]} the second author considers a slightly more general context than usual sum--products setting.
The proof of Theorem \ref{t:BW_R} combined with the arguments of \cite{KS2}, enables one to establish a variant of Theorem~\ref{t:BW_R} as follows.

\begin{theorem}
    Let $A\subset \mathbb C$ be a  set, $\a \neq 0$, and $\d = 1/4$.
    Then there are two disjoint subsets $B$ and $C$ of $A$ such that $A = B\sqcup C$ and
\begin{equation}\label{f:BW_mult1}
    \max\{ \E^{\times} (B), \E^{\times} (\a+C)\} \lesssim  |A|^{3-\d} \,.
\end{equation}
    Further, there are disjoint subsets $B'$ and $C'$ of $A$ such that $A = B'\sqcup C'$ and
\begin{equation}\label{f:BW_mult2}
    \max\{ \E^{+} (B'), \E^{+} (1/ C')\} \lesssim  |A|^{3-\d} \,.
\end{equation}
\label{t:BW_mult}
\end{theorem}

We will further  present a proof of the following  consequence of Theorem~\ref{t:BW_mult}, which improves a result from \cite{s_E_k}.

\begin{corollary}
    Let $A\subset \mathbb C$ be a  set, and let
$$
    R[A] := \left\{ \frac{a_1-a}{a_2-a} ~:~ a,a_1,a_2 \in A,\, a_2 \neq a \right\} \,.
$$
    Then there are two sets $R', R'' \subseteq R[A]$, $|R'|, |R''| \ge |R[A]|/2$
    such that $\E^\times (R') \lesssim |R'|^{3-1/4}$ and $\E^+ (R'') \lesssim |R''|^{3-1/4}$.
\label{c:R_energy}
\end{corollary}

Note that the set $R[A]$ is the set of finite pinned cross-ratios, generated by the projective set $A\cup\{\infty\}$, defined by quadruples $(a,a_1,a_2,\infty)$.

\medskip 
Naturally, an analogue of Theorem \ref{t:BW_mult} over a general field $\F$, following from the proof of Theorem~\ref{t:BW_R} also exists. It is established by combining the arguments of the proof of Theorem \ref{t:BW_R} and  the proof of Proposition 2 in \cite{KS2}\footnote{The proof  merely requires replacing the equation \eqref{trick} in the proof of the forthcoming  Proposition \ref{l:BW_FF}  by the equation corresponding to the energy $\E^\times(\a+A_1)$, rather than $\E^+(A_1)$, where $A_1\subseteq A$ is constructed throughout the proof of Proposition \ref{l:BW_FF}. 
One proceeds by applying Theorem \ref{t:MR} to the latter equation in essentially the same way it is done in the proof of Proposition \ref{l:BW_FF}; the actual application of the theorem can be copied from the proof of Proposition 2 of \cite{AMRS}. }.

\begin{theorem}
  Let $A\subset \F$. If $\F$ has positive characteristic $p$, suppose $|A|\leq p^{5/8}.$  Let $\a \in \F^*$
   and $\d =1/5$.
    Then there are two disjoint subsets $B$ and $C$ of $A$ such that $A = B\sqcup C$ and
$$
    \max\{ \E^{\times} (B), \E^{\times} (\a+C)\} \lesssim  |A|^{3-\d} \,.
$$
\label{t:BW_mult_Fp}
\end{theorem}

We cannot obtain an equally strong analogue for the second statement of Theorem \ref{t:BW_mult} about the set of the reciprocals. Extending the bound (\ref{f:BW_mult2}) to general fields would require
incidence results for hyperbolae rather than affine objects. However, a weaker result  can most likely be derived on the basis of Proposition 14 in \cite{AMRS}. 

In addition, Corollary \ref{t:BW_C} also applies if one replaces the two energies appearing  therein by multiplicative energies of two distinct translates of $A$.



\subsubsection{Few Sums, Many Products}
Our approach also allows for the energy  generalisation of the well-known result of Elekes and Ruzsa, from the paper \cite{ER},  whose title we have used for this subsection. Namely, for $A\subset \mathbb R$ (as well as of $\mathbb C$) one has
$$
|A+A|^4|AA|\gtrsim |A|^6
$$
In fact, we strengthen the above result to an energy-energy inequality,  which is  ``morally'' equivalent to  $|A+A|^3|AA|\gtrsim |A|^5.$ We present the result as an energy inequality and remark that for general fields we do not have an analogue that would be stronger than \eqref{formula1} in Corollary \ref{t:BW_C}.

\begin{theorem}\label{t:fsmp}
Let $A\subset \mathbb C$.  There exist two disjoint subsets $B$ and $C$ of $A$, each of cardinality $\geq |A|/3$, such that
\begin{equation}\label{formula4}
 \E^{\times} (B)  \cdot  \E^+(C)^{3}  \; \lesssim \;|A|^{11}.
\end{equation}
\end{theorem}
However, since the exponents $1,3$ in the above estimate are quite far from one another, estimate \eqref{formula4} 
 beats the non-optimal estimate \eqref{formula3} only if the additive anergy is sufficiently  large.  Besides, the two energies cannot be swapped in Theorem \ref{t:fsmp}.

 \subsubsection{The sum-product estimate over $\mathbb R$}
Theorem \ref{t:fsmp} yields a minor improvement of the stat-of-the-art sum-product exponent over the reals. Konyagin and the second author \cite{KS1}, \cite{KS2} set a new world record towards Conjecture \ref{esc}, having shown that for  a finite set $A$ of reals,
\begin{equation}\label{ksr}
\max(|A+A|,\,|A\cdot A|)\;\gtrsim\; |A|^{\frac{4}{3}+\frac{5}{9813}}.
\end{equation}
This improved the previous best-known exponent $\frac{4}{3}$ obtained some ten years earlier as a result of a graceful and renowned construction by Solymosi \cite{Sol05}, which only relies on the order properties of reals and does not use the Szemer\'edi-Trotter theorem. Within the arguments in \cite{KS1}, \cite{KS2}, the margin by which the value $\frac{4}{3}$ can be beaten depends on the best known estimates apropos of two issues, which can be described as ``few products, many sums'' and ``few sums, many products''. The current approach to both issues that furnishes sufficiently strong estimates is largely based on the Szemer\'edi-Trotter theorem. Dealing with the ``few sums, many products'' side of the coin has been much more successful; this was first done by Elekes and Ruzsa \cite{ER}. Its counterpart proves to be much harder; it is referred by some authors as the {\em weak Erd\H os-Szemer\'edi conjecture}, with the best known estimate stated as \cite[Theorem 12]{KS2}, originating in \cite{s_sumsets}.

Konyagin and the second author proved an energy version of the estimate of Elekes and Ruzsa \cite[Theorem 9]{KS1}; it is slightly weaker than estimate \eqref{formula4} above herein. The following improvement of the sum-product inequality \eqref{ksr} comes after a calculation if one replaces the estimate of \cite[Theorem 9]{KS1} with a variant of estimate \eqref{formula4}. This improves the estimates of \cite[Lemma 18]{KS2} and if one chases through the ensuing \cite[Proof of Theorem 3]{KS2}, the result becomes as follows.

\begin{corollary}\label{objeli} For a finite set $A\subset \mathbb R$, one has
$$
\max(|A+A|,\,|A\cdot A|)\;\gtrsim\; |A|^{\frac{4}{3}+\frac{1}{1509}}.
$$
\end{corollary}


\subsubsection{Balog-Szemer\'edi-Gowers Theorem}
We present some auxiliary results in this short section as a weaker, but arguably less technical alternative to the forthcoming key Propositions \ref{l:BW_FF},  \ref{l:BW_R}. In contrast, this section is about the Balog--Szemer\'{e}di--Gowers theorem \cite{Gow_m}, which our main proofs avoid. We take advantage of the opportunity to present a small but potentially useful modification of one result from Schoen's paper \cite{S_BSzG}.
For modern forms of the Balog-Szemer\'edi-Gowers theorem see, e.g., \cite{BouGar09} and \cite{S_BSzG}.

It is easy to see that the following statement implies the original Balog--Szemer\'{e}di--Gowers theorem.

\begin{theorem}
    Let $(G,+)$ be an abelian group.
    Let $A\subseteq G$ be a set, $K\ge 1$ be a real number, and $k\ge 2$ be an integer.
    Suppose that $\E^+ (A) \ge |A|^3 / K$.
    Then there are sets $A_* \subseteq A$, $P \subseteq A-A$ such that $|A_*| \ge |A|/(8kK)$,
    $|P| \le 8kK|A|$
     and for any $a_1,\dots, a_k \in A_*$ one has
\begin{equation}\label{f:BSzG_Schoen}
    |A \cap (P+a_1) \cap \dots \cap (P+a_k)| \ge \frac{|A|}{4K} \,.
\end{equation}
\label{t:BSzG_Schoen}
\end{theorem}

Theorem~\ref{t:BSzG_Schoen} allows, e.g., for the following  analogue of the forthcoming key Prorosition~\ref{l:BW_R},  which basically stand for the Balog-Wooley decomposition. The estimates are weaker but proofs are simpler.

\begin{proposition}
    Let $A \subset \mathbb C$ be  a  set.
    Then there is $A_1 \subseteq A$ such that $|A_1| \gtrsim \E^\times (A) |A|^{-2}$ and
\begin{equation}\label{f:sp_E_E_1'}
    (\E^{+} (A_1))^2 (\E^\times (A))^9 \lesssim |A|^{32} \,.
\end{equation}
\label{t:sp_E_E'}
\end{proposition}
The special case $\mathbb F=\mathbb C$ in the above formulation indicates the use of the Szemer\'edi-Trotter theorem in the proof below. We challenge an interested reader to formulate a general field analogue, replacing  the use of Theorem \ref{t:ST} by Theorem \ref{t:SzT_Fp}.

\section{Proof of  Theorem \ref{t:BW_R}}

We start out with two intermediate results towards the estimate of Theorem \ref{t:BW_R}: one will later result in $\delta=1/5$ for a general $\F$ and the other in $\delta=1/4$ for $\F=\mathbb C$ in Theorem \ref{t:BW_R}.

\begin{proposition}\label{l:BW_FF}
Let $A\subset \mathbb{F}$, with
$|A|^6\lesssim p^2 \E^\times(A)$
in positive characteristic. Then there is a set $A_{1}\subseteq A$ such that $|A_1|\gtrsim \sqrt{\E^\times(A)/|A|}$ and
\begin{equation}\label{f:ksdf}
\E^{+}(A_1)\lesssim |A_1|^{11/2}|A|^{3/2}(\E^\times(A))^{-3/2}\,.
\end{equation}
The  energies  $\E^\times,\E^+$ in the above statement can be swapped (for some other $A_1$).
\end{proposition}


\begin{proof}
Using the  pigeonhole principle, we choose a dyadic group $P$ of ratios from $A/A$, with approximately some $t$ realisations, which supports at least a fraction of $\frac{1}{\log_2|A|}$ of $\E^\times(A)$. More precisely, there is a set $P\subseteq A/A$ and an integer $t$, such that $|A|^2/(2|A/A|)\leq t\leq |A|$, $t^2|P|\sim \E^\times(A)$, and $t<|A\cap xA|\leq 2t$ for any $x\in P$.  That is, each line through the origin in $\F^2$, with  a slope in $P$ supports about $t$ points of $A\times A$.

Let $S\subseteq A\times A$ be the set of points supported on these lines with slopes in $P$; so $|P|t\leq|S| <2 |P|t$.
Let $\pi_x : S\mapsto A$ be the projection of points of $S$ to the $x$-axis: $\pi_x(s_x,s_y)=s_x$. The projection $\pi_y$ is similarly defined as the projection to the set of ordinates.

Consider the set $A_x$ of abscissae of $S$, that is  $A_x=\pi_x(S)$. By another dyadic pigeonhole argument, we  find a set $A'\subseteq A_x$ of popular abscissae for $S$. There exists $A'\subseteq A_x$, and a number $q'$ such that for every $x'\in A'$, the vertical line through $x'$ supports approximately $q'$ points of $S$ (more precisely $q'<|S\cap \{x=x'\}|\leq 2q'$), and  $|A'|q'\sim |S|$. Observe that $q'\leq \min(|A|,|P|)$, but in the sequel we need an analogue of the slightly stronger inequality $q'\lesssim |A'|$, which is not necessarily true.

So if $q'\leq |A'|$ then we set $A_1=A'$ and $q=q'$. Otherwise we do another dyadic pigeonholing, now by ordinates. I.e., we consider the plane set $S'=S\cap \pi_x^{-1}(A')$ and find a number $q''$ and a set $A''\subseteq \pi_y(S')$ such that $q''<|S'\cap \{y=y''\}|\leq 2q''$ for all $y''\in A''$. In other words,  $A''$ is the set of popular ordinates of the set $S'$ (which in turn is almost as big as $S$): the horizontal line through each $y\in A''$ contains about $q''$ points of $S'$.

We have
$$q'|A'|\sim q''|A''|$$ and since clearly $q''\leq |A'|$, plus since we assume $q'>|A'|$, we must have $|A''|\gtrsim q'>|A'|\geq q''$. We  conclude that $q''< q' \lesssim |A''|$ and set $A_1=A''$ and $q=q''$.

Hence now $|A_1|\gtrsim q$, and so we have
\begin{equation}\label{a1bd} |A_1|^2\gtrsim  q|A_1|\sim |S|\geq |P|t=\frac{|P|t^2}{t}\sim \frac{\E^\times(A)}{t}\geq \frac{\E^\times(A)}{|A|}\,.\end{equation}

Besides, since $|P|t \le |S| \lesssim |A_1|^2$, and $t^2\sim \E^\times(A)/|P|$ we conclude that
\begin{equation}\label{need}
|P|\lesssim |A_1|^4/\E^\times(A).\end{equation}

Without loss of generality let us regard $A_1=A'$, that is the set of popular abscissae, rather than ordinates. We then have by construction of $A'$ a set, each member of which can be represented at least $q$ times as a ratio from $A/P$:
\begin{align}\label{trick}
 \nonumber \E^+(A_1)&=|\{(a,a',b,b')\in A_1^4: a+b=a'+b'\}|\\
&\leq q^{-2}|\{(a,a',p_*,p'_*,\a,\a')\in A_1^2\times P^2\times A^2: a+\a/p_* =a'+\a'/p'_* \}|
\end{align}
Note that by symmetry of $A\times A$,  $P=P^{-1}$, so we can (but do not have to -- this is only a gesture towards the fact that there is no difference as to whether $A_1$ has been taken as $A'$ or $A''$) replace division by $p$ by multiplication.

Consider then the family  of $m=|A_1||P||A|$ planes, with equations $a+p_* x=y+\a'z:\;(a,p_*,\a')\in A_1\times P\times A$ and the same number of points $(x,y,z)\in P\times A_1\times A$. Applying Theorem~\ref{t:MR}  we claim that, in terms of the contribution of the main term $m^{3/2}$ in the estimate of the Theorem
\begin{equation}\label{alm}\E^+(A_1)\ll (|A||P||A_1|)^{3/2}q^{-2},\end{equation}
to be fully justified shortly.

Indeed, the maximum number of collinear points $k$ in the estimate of Theorem ~\ref{t:MR} is bounded by $\max(|A|,|P|)$.
If the maximum equals $|P|$ then to drop the $km$ term in the estimate of Theorem \ref{t:MR} it therefore suffices to show that $|P|\lesssim |A||A_1|$. This is true, since we have established that  $|P|\lesssim |A_1|^2$.
Now if $\max(|A|,|P|) = |A|$ then we need to check $|A| \lesssim |A_1| |P|$.
But $|P| \gtrsim \E^\times (A) / |A_1|^2$ and hence everything follows from a trivial bound $\E^\times (A) \ge |A|^2 \ge |A_1| |A|$.

In positive characteristic Theorem ~\ref{t:MR} is applicable when $ |A_1||P||A|\leq p^2$. By \eqref{need} and the trivial bound $|A_1|\leq|A|$, this will be true given that $|A|^6\lesssim p^2\E^\times(A)$. To strengthen the latter constraint to $|A|^6\leq p^2\E^\times(A)$, as claimed, we proceed as follows.

Suppose, $|A|^6\leq p^2\E^\times(A)$. Partition the set of points $\mathcal P\times A_1\times A$ in $\lesssim 1$ pieces $\{\mathcal P_i\}$ (say, by partitioning $A$), whose size differs by at most an absolute constant factor, and such that each $|\mathcal P_i|\leq p^2$. The number of solutions of the  equation in the second line of \eqref{trick} is the sum, over $1\leq i\lesssim 1$, of the number of incidences between the above $m$ planes and the point set $\mathcal P_i$. By Theorem \ref{t:MR} it is $O(m^{3/2})$, for each $i$.

Thus the summation over $i\lesssim 1$ results only in the change of the power of $\log|A|$ hidden in the  incidence estimate \eqref{alm}, that is if $|A|^6\leq p^2\E^\times(A)$ the estimate \eqref{alm} is true, with a different factor of $\log|A|$ hidden in the $\lesssim$ symbol.

We pass from  formula \eqref{alm} to \eqref{f:ksdf} by setting  $q \sim  |P|t/|A_1|$, $\E^\times(A)\sim |P|t^2$ and using \eqref{need} to bound the remaining $\sqrt{|P|}$ in the numerator. 

We have established the formula \eqref{f:ksdf} as it is. The fact that energies can be swapped follow by taking the above set $P$ as a subset of $A-A,$ rather than of $A/A$. and repeat the argument. The only modification is that the equivalent of \eqref{trick} will now deal with $\E^\times(A_1)$, followed by the equivalent of \eqref{alm}, as worked out explicitly in the proof of Proposition 1 in \cite{AMRS}.

This completes the proof of Proposition \ref{l:BW_FF}.
\end{proof}

In the special case of the real or complex field we have a slightly stronger result via the Szemer\'edi-Trotter theorem.

\begin{proposition}\label{l:BW_R}
Let $A\subset \mathbb{C}$ be a  set. Then there is $A_{1}\subseteq A$ such that $|A_1|^2\gtrsim \E^\times(A)|A|^{-1}$ and \begin{equation}\E^{+}(A_1)\E^\times(A)\lesssim {|A_1|}^{9/2}|A| \,.\label{f:stuff}\end{equation}
The  energies  $\E^\times,\E^+$ in the above statement can be swapped (for some other $A_1$).
\end{proposition}

\begin{proof} One repeats the pigeonholing  arguments in the proof of Proposition~\ref{l:BW_FF} with the same notations $t,P,q,A_1$ up to and inclusive of  the estimate \eqref{need}.  Without loss of generality we assume that $A_1$ is the set of popular abscissae $A'$ in the construction of Proposition~\ref{l:BW_FF}.

Using the notation of \cite{KS1}, we observe that
$$A_1\subset \text{Sym}_q(A,P):=\{x:|A\cap xP^{-1}|\geq q\}\,,$$
which means that, by construction, each member of $A_1$ can be represented at least $q$ times as a ratio from $A/P$ (or a product $AP$, since $P=P^{-1}$).

Using Lemma 13 and Corollary 11 of  \cite{KS2} we conclude\footnote{For a reader not willing to consult \cite{KS2} we sketch the argument that goes back to the paper of Elekes \cite{E} and is similar to the one in the conclusion of the proof of the Proposition \ref{l:BW_FF}. Instead of  Theorem \ref{t:MR} one uses the Szemer\'edi-Trotter theorem to estimate the cardinality of the set $S_\tau$ of sums $s=a+b$ in $A_1+A_1$ with approximately $1\leq\tau\leq |A_1|$ realisations. One can rewrite $s=a+ \a p,$ with $\a\in A$, $p\in P^{-1}=P$, and therefore there are at least $q\tau$ incidences between the point set $A\times S_\tau$ and a set of $|P||A_1|$ lines.  Theorem \ref{t:ST} then gives the upper bound on $|S_\tau|$; recycling this into an energy estimate is a standard exercise. \label{ftn}}
that
\begin{equation}\label{shen} \E^+(A_1) \ll d_*(A)^{1/2} |A_1|^{5/2},\end{equation} where the parameter (see \cite{KS2})
$$d_*(A_1) \le \frac{|A|^2|P|^2}{q^3 |A_1|}.$$
Using the relations $|A_1|q\sim |P|t\lesssim |A_1|^2$ from the proof of Proposition \ref{l:BW_FF} we obtain
$$d_*(A_1) \sim \frac{|A|^2|A_1|^2}{|P|t^3} \sim\frac{|A|^2|A_1|^2|P|t}{{\E^\times (A)}^2}\sim\frac{|A|^2|A_1|^4}{{\E^\times (A)}^2}\,.$$

Substituting the latter estimate into \eqref{shen} completes the proof of Proposition \ref{l:BW_R}. The fact that the energies $\E^\times,\E^+$ can be swapped follows taking $P\subseteq A-A$ instead of $A/A$ and repeating the argument. See also Theorem 20 in \cite{KS2}. 
\end{proof}

\begin{remark} \label{rrem} The same proof can be easily modified to an application of Theorem \ref{t:SzT_Fp} here instead of Theorem \ref{t:ST}. The latter theorem would also enable one to obtain decomposition estimates involving higher energies, along the lines of those obtained by the second author \cite{s_E_k} over the reals. 

 We note, however, that the proofs of Propositions \ref{l:BW_FF}, \ref{l:BW_R} use the symmetry between the $x$ and $y$-axes and are not applicable to the quantities $\mathrm{D}^{+}, \mathrm{D}^\times$, studied by the second author in \cite{s_E_k}. 
\end{remark}

To conclude the proof of Theorem \ref{t:BW_R}, the above intermediate results are iterated via a simple lemma: see, e.g.,  \cite{TaoVubook}.

\begin{lemma}\label{1/4ineq}
Let $A_1,\dots,A_n$ be  subsets of an abelian group.
Then
$$
    \left(\E^{+} \left(\bigcup_{i=1}^n A_i\right)\right)^{1/4} \le \sum_{i=1}^n (\E^{+} (A_i))^{1/4}.
$$
\end{lemma}
Observe that  Propositions \ref{l:BW_FF}, \ref{l:BW_R} imply a weaker version of Theorem \ref{t:BW_R}, where one replaces $\max(B,C)$ with $\min(B,C)$.


\subsubsection{Conclusion of the proof of Theorem \ref{t:BW_R} and proofs of Corollary \ref{t:BW_C}}

\begin{proof}[Conclusion of the proof of Theorem \ref{t:BW_R}] Suppose $\E^\times(A)\lesssim |A|^{3-\delta}$ or there is nothing to prove: we can always take a small $B$ such that $|B|^3< |A|^{3-\delta}$ and $C=A\setminus B$.

Heuristically, we use Proposition \ref{l:BW_FF} for a general $\F$ and Proposition \ref{l:BW_R} in the special case $\F=\mathbb R$ or $\mathbb C$ to pull out from $A$, one by one, subsets $A_1$ with a small additive energy and stop once the energy of the remainder $C$ of $A$ will become smaller than  $|A|^{3-\delta}.$ This is bound to happen if $C$ becomes sufficiently small relative to $|A|$. The above lemmas are going to guarantee that $\E^+(B=A\setminus C)\lesssim |A|^{3-\delta}$.

Formally, let $M\geq 1$ be a parameter which we  choose later, assume that  $\E^\times(A)\leq |A|^{3}/M$ and
$|A|^3\leq p^2 / M$
in positive characteristic.

We construct a decreasing sequence of sets $C_1=A \supseteq C_2 \supseteq \dots \supseteq C_k$ and an increasing sequence of sets $B_0 = \emptyset \subseteq B_1 \subseteq \dots \subseteq B_{k-1} \subseteq A$ such that
    for any $j=1,2,\dots, k$ the sets $C_j$ and $B_{j-1}$  are disjoint and moreover $A = C_j \sqcup B_{j-1}$.
    If at some step $j$ we have $\E^{\times} (C_j) \leq |A|^3 / M$, we stop and set
    $C=C_j$, $B = B_{j-1}$, and $k=j-1$.
   Else, we have $\E^{\times} (C_j) > |A|^3 / M$. We apply Proposition \ref{l:BW_FF}
to the set $C_j$, finding the subset $D_j$ of $C_j$ such that
\begin{equation}|D_j|^2 \gtrsim \frac{\E^\times(C_j)}{|C_j|}>\frac{|A|^3}{M|C_j|}\geq\frac{|A|^2}{M}\label{hm}\end{equation}
and
\begin{equation} \label{dest}
    \E^{+} (D_j) \lesssim \frac{|D_j|^{11/2} |C_j|^{3/2}}{(\E^\times(C_j))^{3/2}}\leq\frac{M^{3/2}|D_j|^{11/2}|C_j|^{3/2}}{|A|^{9/2}}<|D_j|^{11/2}M^{3/2}|A|^{-3} \,.
\end{equation}
    After that we put $C_{j+1} = C_j \setminus D_j$, $B_j = B_{j-1} \sqcup D_j$ and repeat the procedure.
    In view of the uniform lower bound \eqref{hm} on $|D_j|$, the process will terminate, as $|C_j|$ decreases, after say $k$ iterations, when $\E(C=C_{k+1})\leq |A|^3/M$. We set $B=B_k = \bigsqcup_{j=1}^k D_j$ and $C=C_{k+1}$, so $ A=B\sqcup C$.

Trivially $|B| =\sum_{j=1}^{k}|D_j|\leq |A|$.\\

Then, using Lemma \ref{1/4ineq} and the bound \eqref{f:ksdf} of Proposition \ref{l:BW_FF} we get
\begin{align}\label{formula}
    \E^{+} (B) &\lesssim M^{3/2} |A|^{-3} \left( \sum_{j=1}^{k} |D_j|^{11/8}\right)^4\\ \nonumber
         &\leq M^{3/2}|A|^{-3}\max_j\{|D_j|\}^{3/2}\left(\sum_{j=1}^k|D_j|\right)^4\\  \nonumber
         &\leq  M^{3/2}|A|^{5/2}
 \,.
\end{align}
Optimising over $M$, with  $\E^{*} (C) \leq |A|^3/M$,  that is choosing $M=|A|^{1/5}$, we obtain the result for a general $F$. In particular, the constraint in terms of $p$ in positive characteristic boils down to $|A|^6<|A|^{3-1/5}p^2$, so $|A|\leq p^{5/8}.$

In the special case of $\F=\mathbb C$ the analogue of the estimates \eqref{dest} and \eqref{formula} comes from using \eqref{f:stuff}, rather than \eqref{f:ksdf}. Namely, we have
$$
 \E^{+} (D_j) \lesssim \frac{|D_j|^{9/2} |C_j|}{\E^\times(C_j)}\leq \frac{ M |D_j|^{9/2}} {|A|^{2} }\,
$$
and
$$    \E^{+} (B) \; \lesssim \; M |A|^{-2} \left( \sum_{j=1}^{k} |D_j|^{9/8}\right)^4
         \;\leq \; M |A|^{5/2}.
$$
Optimising with  $\E^{*} (C) \leq |A|^3/M$, yields $M=|A|^{1/4}$.
This proves Theorem  \ref{t:BW_R}. \end{proof}

\begin{proof}[Proof of Corollary \ref{t:BW_C}]
Observe that if $A$ has a subset $C$ with $|C|\geq 2|A|/3$ and  $\E^\times(C)\leq |C|^{8/3} $, there is nothing to prove: just take $B$ as half of $C$, with the trivial bound $\E^+(B)\leq |A|^3$ and rename $C$ as the other half.

Otherwise we repeat the argument in the preceding proof, finishing it at the first instance when either $\E^\times(C=C_{k+1}) <|A|^{8/3}$ or $|B=B_k|>|A|/3$. Without loss of generality we can assume that  in addition to the estimates \eqref{dest} and \eqref{formula} we have, say $|D_j|<|A|/100$ for every $j$ (by partitioning $D_j$ if necessary). Since the sequence of sets $\{C_j\}$ is decreasing, we have $\E^\times(C) \leq \E^\times(C_j),\;\forall j\leq k+1.$ Hence, in view of \eqref{f:ksdf}, the calculation \eqref{formula} becomes

$$
 \E^{+} (B)   \E^\times(C)^{3/2}  \; \lesssim \; |A|^{3/2} \left( \sum_{j=1}^{k} |D_j|^{11/8}\right)^4 \;\leq |A|^7,
$$
thus proving \eqref{formula1}.

Note that Proposition \ref{l:BW_FF} does indeed apply on each step in positive characteristic, for the condition $|A|^6 \leq p^2\E^\times(C)$ is satisfied, since we have assumed $\E^\times(C)\geq |A|^{8/3}$ and $|A|\leq p^{3/5}.$

To derive \eqref{formula2} of the Corollary from \eqref{formula1} we observe that according to the above proofs one can swap the two energies in \eqref{formula1}. We do this for the set $B$. Namely $B$ gets partitioned into $B'$ and $C'$, each of size at least $|B|/3$, such that
$$
\E^+(B')^{3/2} \E^\times(C') \lesssim |A|^7.
$$
Multiplying the latter two estimates we obtain
$$
(\E^\times(B) \E^+(B') )^{3/2} \cdot (\E^+(C)\E^\times(C') )\lesssim  |A|^{14}.
$$
If $\E^\times(B) \E^+(B')  \leq  \E^+(C)\E^\times(C')$,  we get
$$
(\E^\times(B) \E^+(B'))^{5/2} \lesssim |A|^{14},
$$
(and one can replace $B$ by $C'$, disjoint from $B'$) otherwise we get the same inequality involving the disjoint $C$ and $C'$.

Renaming the two subsets in question as $B$ and $C$ finishes the proof, the price we've paid is just that $B$ and $C$ no longer partition $A$.
\end{proof}


\section{Proof of Theorem \ref{t:fsmp} and Corollary \ref{objeli}}
We give a detailed sketch of the proof of Theorem \ref{t:fsmp}, which largely repeats our previous arguments, the key benefit being derived from using an estimate from \cite{MR-NS}.
\begin{proof}{Sketch of proof}

We invoke the additive version of the construction in the proof of Proposition \ref{l:BW_FF}, with the same notations, to derive a version of the proposition, given by  the forthcoming estimate \eqref{vers}. $P$ is now  a dyadic group of popular sums, with approximately (that is up to a constant factor) $t$ realisations each, that supports $|P|t^2 \gtrsim \E^+$ of additive energy. $S$ is the corresponding subset of $A\times A$. $A_1$ is the set of popular abscissae (ordinates) for $S$, with approximately $q$ realisations, and $q\lesssim |A_1|$. Moreover, $|S|t\lesssim |A_1|^2$ and $|A_1|\gtrsim \sqrt{\E^+(A)/|A|}.$

Assuming that $A$ does not contain zero, consider the multiplicative energy equation
$$ a/b = a'/b' :  \,a,b,a',b'\in A_1.$$

For the left-hand side there are  approximately $q$ choices to add some $c\in A$ in the numerator and some $d\in A$ in the denominator to replace it with

$$\frac{(a+c) - c}{ (b+d) - d}  = \frac{ s-c }{r-d}, $$
where $s,r \in P$. Thus $\E^\times(A_1)$ is bounded by a constant, times $q^{-4}$, times the number of solutions of
$$
\frac{s-c}{r-d} = \frac{s'-c'}{r'-d'} \neq 0,\infty:  \;(s,r,s',r')\in P^4,\; (c,d,c',d')\in A^4. $$
The latter equation has been studied, in particular in the paper of Murphy, Roche-Newton, and the second author  \cite{MR-NS} (see Lemma 2.5 therein) which proves the upper bond $\lesssim |P|^3|A|^3$ for the number of solutions.

Hence, once $|A_1|^2\gtrsim q|A_1| \gtrsim |P|t$ we have
\begin{equation}\label{vers}
\E^\times(A_1) \lesssim q^{-4} |P|^3 |A|^3 \lesssim  \frac{|P|^3 |A|^3} {|P|^4t^4/|A_1|^4}  = |A_1|^4|A|^3\frac{|P|^2t^2}{|P|^3t^6}\lesssim  \frac{ |A_1|^{8}|A|^3}{{\E^+(A)}^3}.
\end{equation}

A straightforward adaptation of the iterative argument in the first two passages of the proof of Corollary \ref{t:BW_C} to the latter estimate completes the proof.
\end{proof}

\subsection{Proof of Corollary \ref{objeli}}
Observe that inequality \eqref{vers} implies, by Cauchy-Schwarz, that 
\begin{equation}\label{subs}
|A_1\cdot A_1|, \;|A_1/A_1| \gtrsim \frac{{\E^+(A)}^3}{ |A_1|^{4}|A|^3}\geq \frac{{\E^+(A)}^3}{ |A|^7},
\end{equation}
hence the same lower bound for the supersets $|A\cdot A|, \;|A/A|$.
In \cite[Theorem 9]{KS1} a weaker estimate was established:
$$
|A\cdot A|, \;|A/A| \gtrsim \frac{{\E^+(A)}^4}{|A|^{10}}.
$$
The latter estimate was used to obtain \cite[Lemma 12, estimates (34), (36)]{KS1}, restated as \cite[Lemma 18, estimates (22), (24)]{KS2}. If one uses \eqref{subs} instead, this improves the term $L^{-16}$ in these estimates to $L^{-12}$. Recalculating \cite[Proof of Theorem 3 from (26) on]{KS2} yields the new sum-product exponent $\frac{4}{3}+\frac{1}{1509}$ as claimed. \qed

\section{Proof of Corollary~\ref{c:R_energy}}
\begin{proof}
Put $R=R[A]$, $R^* = R \setminus \{0\}$, and $\d =1/4$.
Using Theorem \ref{t:BW_mult}, we find $B,C \subseteq R$ such that $R = B\sqcup C$ and
$$
    \max\{ \E^{\times} (B), \E^{\times} (C-1)\} \lesssim  |R|^{3-\d} \,.
$$
If $|B|\ge |R|/2$ then we are done.
Suppose not.
Then $|C| \ge |R|/2$ and in view of formula $R=1-R$, see \cite{Shkredov_R[A]}, we obtain that
$C' := 1-C \subseteq R$, $|C'| = |C| \ge |R|/2$ and
$$
    \E^{\times} (C') = \E^{\times} (1-C) = \E^\times (C-1) \lesssim  |R|^{3-\d} \,.
$$
So, putting $R'$ equals $B$ or $C'$, we obtain the result.
To find the set $R''$ note that $(R^*)^{-1} = R^*$ and use the second part of Theorem \ref{t:BW_mult}.
\end{proof}

The same proof allows us to find a subset $A'_s$ of the set $A_s \cup (-A_s)$, $A_s = A\cap (A+s)$, $s\in A-A$,
$A \subset \mathbb C$ of cardinality $|A_s|/2$ such that $\E^\times (A'_s) \lesssim |A'_s|^{3-1/4}$.
This question is a dual of one which appeared in \cite{KS1}, \cite{KS2}.
The same result holds for some multiplicative analogue of the sets $A_s$, namely, $A^*_s = A\cap (s/A)$, $s\in AA$.

\section{Proof of Theorem~\ref{thm:had}}

We now turn to the proof of Theorem~\ref{thm:had}. We shall apply Corollary \ref{t:BW_C} to the set $A$: it applies when $|A|\leq p^{3/5}$, which we may assume, passing to a subset if $A$ is too big.

So, if $|A|\leq p^{3/5}$, Corollary \ref{t:BW_C} gives us two positive proportion disjoint subsets $B,C$ of $A$, whose energies satisfy the estimate \eqref{formula1} and whereon we consider equation \eqref{e:had}.

Let us denote
$$
Q= \left|
        \left\{\frac{ab-c}{a-d}:\;a,b,c,d\in A, a\neq d\right\}\right|.
$$

By Cauchy-Schwarz
\begin{align}\label{e:hadcs}
\Theta(|A|^8) = |B|^4 |C|^4  & \leq  \; Q \left(\sum_{x\in \mathbb{F}_{p}}\left|\left\{x=\frac{ab-c}{a-d}: \,a,b\in B;\,c,d\in C\right\}\right|\right)^2 .
\end{align}

Let us isolate the case $x=0$. This means $ab=c$ and $a'b'=c'$, so the trivial bound for this is $|A|^6$.

We can then denote
$$\mathcal E= \sum_{x\in \mathbb{F}^*_{p}}\left|\left\{x=\frac{ab-c}{a-d}: \,a,b\in B;\,c,d\in C\right\}\right|^2$$ and assuming $|A|>p^{1/2}$ rewrite \eqref{e:hadcs} as
\begin{equation}
|A|^8\ll Q\mathcal E,
\label{easy}
\end{equation}
for trivially $Q\leq p$. We further aim to find the upper bound on $\mathcal E$.

Rearranging and applying Cauchy-Schwarz once more yields
\begin{equation}\label{int}
\mathcal E = \sum_{x\neq 0}|\{a(b-x)=c-dx\}|^2\leq \sum_{x\neq 0}\E^\times(B,x+B)\E^+(C,xC)\,.
\end{equation}
We evaluate these energy terms using an argument of \cite{Petridis}, modifying it to suit our needs. Although \cite{Petridis} considers  the term containing $\E^+(C,xC)$ only, the analysis applies almost verbatim to the  quantity $\E^\times(B,x+B)$ as well.

\medskip
In the following lemmata, $A$ denotes a dummy set, to be replaced, respectively, by $C$ and $B$.
We first quote a well known fact, which has been in the literature since ``generic projections'' in  \cite{BKT}, recorded as Lemma 3 in \cite{Petridis}.
\begin{lemma}\label{ensum} Let $A\subseteq \F_p$ with $|A|>p^{1/2}$. Then
$$
\sum_{x\in \F_p^*} \E^+(A,xA), \; \sum_{x\in \F_p^*} \E^\times(A,x+A) \ll |A|^4.
$$\end{lemma}
Heuristically the above statement means that knowing, say $a,b,a',b'$ in the equation $a(b-x)=a'(b'-x)$ generically defines $x$, the constraint $|A|>p^{1/2}$ taking care of degeneracies. So the latter $\ll$ estimate is, in fact, an asymptotic identity.

Besides, by Cauchy-Schwarz, for every $x\in \F_p^*$ we have
\begin{equation} \E^+(A,xA),\; \E^\times(A,x+A) \geq |A|^4/p.\label{trb}\end{equation}
Indeed, we simply use $p$ as the upper bound for, say $|A+xA|$.

Combining the formula \eqref{int} -- where we set $\E^\times(B,x+B)= (\E^\times(B,x+B)-|B|^4/p) + |B|^4/p$  and similarly for $\E^+(C,xC)$ -- with Lemma \ref{ensum} we have, assuming $|A|>p^{1/2}$ and using the H\"older inequality we obtain:
\begin{align}\label{long}
\mathcal E&\ll \frac{|A|^8}{p} + \sum_{x\neq 0}\left(\E^+(C,xC)-\frac{|C|^4}{p}\right) \left(\E^\times(B,x+B)-\frac{|B|^4}{p}\right)\\ \nonumber
& \leq \frac{|A|^8}{p} +\left( \sum_{x\neq 0}\left(\E^+(C,xC)-\frac{|C|^4}{p}\right)^{5/3}\right)^{3/5} \left(\sum_{x\neq 0}\left(\E^\times(B,x+B)-\frac{|B|^4}{p}\right)^{5/2}\right)^{2/5}.
\end{align}
The choice of exponents in the H\"older inequality has been made to conform with the estimate \eqref{formula1} in the sequel.


The main part of the argument is the following proposition, an analogue of Proposition~8 in \cite{Petridis}.
\begin{proposition}\label{p:had} Let $A\subseteq \mathbb{F}_p$, with $p^{1/2}<|A|\leq p^{2/3}$ and $s\in (0,3)$. Then
$$
\sum_{x\neq 0}\left(\E^+(A,xA)-\frac{|A|^{4}}{p}\right)^{1+s}=O\left(p^{1-\frac{1}{3}s}\E^+(A)^{\frac{2}{3}s}|A|^{2+\frac{4}{3}s}\right)$$
and $$
\sum_{x\neq 0}\left(\E^\times(A,x+A)-\frac{|A|^{4}}{p}\right)^{1+s}=O\left(p^{1-\frac{1}{3}s}\E^\times(A)^{\frac{2}{3}s}|A|^{2+\frac{4}{3}s}\right).$$
\end{proposition}

Proposition \ref{p:had} relies on the following lemma, the analogue of Theorem 2 in \cite{Petridis}, which follows from Theorem \ref{t:MR}.

\begin{lemma}\label{inccount}
Let $A\subseteq \mathbb{F}_p,$ $X\subseteq \mathbb{F}_p^*$.
Suppose
$|X|=O(|A|^2)$ and $|A|^2|X|=O(p^2)$. Then
$$\sum_{x\in X}\E^+(A,xA) \ll \E^+(A)^{1/2}       |A|^{3/2}|X|^{3/4}$$
and
$$\sum_{x\in X}\E^\times(A,x+A) \ll \E^\times(A)^{1/2}|A|^{3/2}|X|^{3/4} \,.$$
\end{lemma}

To keep the exposition more self--contained we sketch the proof of Lemma \ref{inccount}, dealing with its second estimate: the first one was the result of Theorem 2 in \cite{Petridis}, and the two proofs are in essence identical.
\begin{proof}(Sketch of proof.)
Let $Y=\sum_{x\in X}\E^\times(A,x+A)$, that is the number of solutions of the equation
$$
a(x+b)=c(x+d): \;x\in X,\,a,b,c,d\in A.
$$
It is easy to see that the condition $|X|=O(|A|^2)$ implies that the number of trivial solutions of the last equation (e.g. $a=c=0$ and $x,b,d$ are any or $x=-b$, $c=0$ and $a,d$ are arbitrary) is negligible.
Rearranging and applying Cauchy-Schwarz we get  $Y\ll (\E^\times(A)Z)^{1/2}$, where $Z$ is the number of solutions of the equation
$$
\frac{x+b}{x+d} = \frac{x'+b'}{x'+d'}:\;x,x'\in X, b,d,b',d' \in A.
$$
The number of nontrivial solutions of the latter equation can be estimated by using Theorem 19 in \cite{AMRS}, which is a particular case of Theorem \ref{t:MR}. This was done, in particular, in the proof of Corollary 8 of the latter paper. We nonetheless briefly show how.

Consider a family of planes in $\F_p^3$, with equations, in coordinates $(x_1,x_2,x_3)$ not to be confused with $x,x'\in X$, as follows:
$$
\frac{1}{x+d} x_1 - x_2 - b'x_3 +  \frac{x}{x+d} = 0,
$$
and their incidences with points $(x_1,x_2,x_3) = (b, \frac{x'}{x'+d'}, \frac{1}{x'+d'})$. It is easy to verify, see \cite{AMRS},  that there are $m=|X||A|^2$ planes and points. The constraints of Lemma \ref{inccount} ensure that the main term $m^{3/2}$ dominates in the estimate Theorem \ref{t:MR} and the theorem is applicable. Thereupon, the second estimate of Lemma \ref{inccount} follows.
\end{proof}

\begin{remark}
    Using methods from \cite{MR-NS} (see the proofs of Lemmas 2.3, 2.4)
    one can prove another variant of Lemma \ref{inccount}, namely,
    for any $A\subseteq \mathbb{F}_p$ and $X\subseteq \mathbb{F}_p^*$ with $|A| \leq p^{2/3}$
    the following holds
$$\sum_{x\in X}\E^+(A,xA) \ll |A|^{13/4} |X|^{1/2} \,,$$
and
$$\sum_{x\in X}\E^\times(A,x+A) \ll |A|^{13/4}|X|^{1/2} \,.$$
\end{remark}

Lemma \ref{inccount} can be restated in the standard way, similar to the well-known restatement of the Szemer\'edi-Trotter theorem, Theorem \ref{t:ST} here, as $O\left(\frac{n^2}{\tau^3} + \frac{n}{\tau}\right)$ as the upper bound on the number of $\tau$-rich lines. See Lemma~4 of \cite{Petridis}.
We quote an auxiliary Lemma from \ref{BKT} which appears (and is proved) explicitly in \cite{Petridis} as Lemma 3.

\begin{lemma}\label{BKT}
For every set $X \subseteq \F_p^*$ we have the following inequality
$$  \sum_{x\in X} \left(\E^+(A, xA) -\frac{ |A|^4}{p}\right),\, \sum_{x\in X} \left(\E^\times(A, x+A) -\frac{ |A|^4}{p}\right) \; \leq \; p|A|^2\,.$$

\end{lemma}

\begin{lemma}\label{proplemma4}
Let $A\subseteq \mathbb{F}_{p}$, $|A|> p^{1/2}$ and $1\leq K\leq p\frac{\E^+(A)}{2|A|^{4}}$. \\
Then the number of $x\in \mathbb{F}_p^*$ such that $\E^+(A,xA)>\frac{\E^+(A)}{K}$  is $O\left(K^{4}\frac{|A|^{6}}{\E^+(A)^{2}}\right)$.
A similar bound holds for the number of $x\in \mathbb{F}_p^*$ such that $\E^+(A,x+A)>\frac{\E^\times(A)}{K}$.
\end{lemma}
\begin{proof}
Note that $K$ is well defined,  by \eqref{trb}.
Let $X$ be the set of $x$ in question.\\
Firstly we show that $|A|^{2}|X|=O(p^{2})$. From Lemma~\ref{BKT} we know that
$$p|A|^2\geq \sum_{x\in X}\left(\E^+(A,xA)-\frac{|A|^4}{p}\right)\geq\sum_{x\in X}\left(\frac{\E^+(A)}{K}-\frac{|A|^4}{p}\right)\,.$$
Hence we have
$$p|A|^2\geq |X|\left(\frac{\E^+(A)}{K}-\frac{|A|^4}{p}\right),$$
and because of our choice of $K$, the first term dominates. So $|X|\leq \frac{2p |A|^2 K}{\E^+(A)}$; substitution yields $|A|^2|X|=O(p^2)$ and $|X| = O(|A|^2)$. \\
Next, we use Lemma~\ref{inccount} to obtain
$$\sum_{x\in X}\E^+(A,xA)=O(\E^+(A)^{1/2}|A|^{3/2}|X|^{3/4})\,.$$
All that remains is to evaluate the following: $$\frac{|X|\E^+(A)}{K}\leq \sum_{x\in X}\E^+(A,xA)=O(\E^+(A)^{1/2}|A|^{3/2}|X|^{3/4}).$$ One rearranges to obtain the desired bound on $|X|$.

The multiplicative energy case is identical.
\end{proof}

Thus we pass from having to satisfy the conditions of Lemma~\ref{inccount} to meeting those of Lemma~\ref{proplemma4}. We now prove Proposition~\ref{p:had}; as the proofs for both statements are the same, we prove the first one only.
\begin{proof}[Proof of Proposition~\ref{p:had}] Note that by Cauchy-Schwarz, for every $x\neq 0$,  we have $\E^+ (A,xA)\leq \E^+(A).$

Set $M=(p \E^+ (A) / 8|A|^4)^{1/3}$, so $1/2\leq M\leq  p \E^+ (A) /(2|A|^4)$ by \eqref{trb}.
Since also $p^{1/2}<|A|\leq p^{2/3}$,  for any $1\leq K \leq M$ the conditions of Lemma \ref{proplemma4} are satisfied.

For `small' energies when $\E^+ (A,xA)\leq \E^+(A)/M$, we rely on \eqref{trb} which ensures no sign alterations and proceed with a trivial inequality and Lemma~\ref{BKT} to obtain:
\begin{align*}
\sum\left( \E^+(A,xA)-\frac{\aa^4}{p}\right)^{1+s}  & \leq \sum\left(\frac{\E^+(A)}{M}\right)^s\left( \E^+(A,xA)-\frac{\aa^4}{p}\right) \\ &
\leq p\aa^2\left(\frac{\E^+(A)}{M}\right)^s, \end{align*}
where both sums are taken over the set $\{x\neq 0: \E^+(A,xA)\leq \E^+(A)/M\}$.
Note that if $M<1$ then the inequalities $\E^+(A,xA)-\frac{\aa^4}{p} \leq \E^+ (A,xA)\leq \E^+(A)/M$ hold by trivial reasons.

For `large' $\E^+(A,xA)$ we use Lemma \ref{proplemma4} and a dyadic argument. Let
$$X_i:=\{x: 2^i \E^+(A)/M < \E^+(A,xA)\leq \min(2^{i+1}\E^+(A)/M, \E^+(A)\}.$$

For $0\leq i\ll \log M$, set $2^{-i}M=K_i$, so $1\leq K_i \leq M.$

By Lemma~\eqref{proplemma4} we know that $|X_{i}|=O\left(2^{-4i}\frac{M^4|A|^{6}}{\E^{+}(A)^{2}}\right)$. So
\begin{align*}
\sum_{\substack{\E^+(A,xA)>\E^+(A)/M}}\left(\E^+(A,xA)-\frac{|A|^{4}}{p}\right)^{1+s}&\leq \; \sum_i\sum_{x\in X_{i}}\E^+(A,xA)^{1+s}\\
&\ll\; \sum_i|X_i|\E^+(A)^{1+s}2^{(1+s)i} M^{-(1+s)}  \\ & \ll \frac{ M^{3-s} |A|^{6}}{\E^+(A)^{1-s}}.
\end{align*}
Here we have used that $s\in (0,3)$ to sum the geometric progression in $i$, so the constant hidden in the last inequality depends on $s$. In view of the choice of $M$, independent of $s$, the latter two estimates match, concluding the proof of Proposition~\ref{p:had}.
\end{proof}

\bigskip

To conclude the proof of Theorem \ref{thm:had} we observe that by    Proposition \ref{p:had}, applied  to the estimate \eqref{long} we have
\begin{equation}\label{almost}
\mathcal E \;  \ll  \; |A|^8/p + p^{2/3} |A|^{10/3} [  \E^+( C) \E^\times (B)^{3/2} ] ^{4/15},
\end{equation}
given that $|A|\leq p^{3/5}$. Applying \eqref{formula1} of Corollary \ref{t:BW_C} (with the notations $B$ and $C$ reversed) it is immediate to conclude that the first term in the estimate dominates if
$|A|\geq p^{25/42} \log^K|A|$, where $K$ depends only on the power of $\log|A|$ hidden in the estimate \eqref{formula1}.  Note that $25/42<3/5$, so the upper bound on $|A|$ to make Corollary  \ref{t:BW_C} applicable has been satisfied.

For such $A$ the number of solutions of equation \eqref{e:had} is $O(|A|^8/p)$, which completes the proof of Theorem \ref{thm:had}. \qed


\section{Proof of Theorem~\ref{t:BSzG_Schoen} and ensuing statements}
\label{sec:further}
We prove Theorem~\ref{t:BSzG_Schoen}.

{\bf Proof.}
    Let $\E= \E^+(A) \geq \aa^3/K$.
    Put $A_s = A\cap (A-s)$.
    We have
\begin{equation}\label{f:energy_BSzG}
    \E = \sum_{s\in G} |A_s|^2 = \sum_s \sum_{x,y} A_s (x) A_s (y) = \sum_s \sum_{x,y} A(x) A(y) A(x+s) A(y+s)  \,.
\end{equation}
    Let  $\eps \in (0,1)$ be a real number which we will choose later.
    Let us put
$$
    P = P_\eps = \left\{ s ~:~ |A_s| \ge \frac{\eps \aa}{2K} \right\} \,.
$$
    Clearly, the set $P_\eps$ is symmetric for any $\eps$ and $|P_\eps| \le 2K \eps^{-1} |A|$.
    As
$$
    \sum_s\, \sum_{x-y\notin P_\eps} A_s (x) A_s (y) = \sum_{x-y\notin P_\eps} A(x) A(y) |A_{x-y}|
        <
            \frac{\eps \aa^3}{2K}
$$
then by combining the last estimate with (\ref{f:energy_BSzG}), we get
$$
    \sum_{s~:~ |A_s| \ge \aa/2K}\, \sum_{x-y\in P_\eps} A_s (x) A_s (y)
        -
            \eps^{-1} \sum_{s~:~ |A_s| \ge \aa/2K}\, \sum_{x-y\notin P_\eps} A_s (x) A_s (y)
                > 0 \,.
$$
It follows that there is $s$ with $|A_s| \ge \aa/(2K)$ satisfying,
$$
    \eps \sum_{x-y\in P_\eps} A_s (x) A_s (y) > \sum_{x-y\notin P_\eps} A_s (x) A_s (y) \,.
$$
In other words
\begin{equation}\label{tmp:17.07.2014_1}
    \sum_{x-y\in P_\eps} A_s (x) A_s (y) > (1-\eps) |A_s|^2 \,.
\end{equation}
Now let us consider a non--oriented graph (with loops)  $\mathcal{G}=(\mathcal{V},\mathcal{E})$ with the vertex set equal to $A_s$ such that its vertices  $x,y$ are connected iff $x-y\in P_\eps$.
It is easy to see from (\ref{tmp:17.07.2014_1}) that  $|\mathcal{E}| > (1-\eps) |A_s|^2$.
Put
$$
    A_* = \{ v \in V ~:~ \deg v \ge (1-2\eps) |\mathcal{V}| \} \subseteq A_s \,.
$$
Inequality  (\ref{tmp:17.07.2014_1}) implies $|A_*| \ge \eps |\mathcal{V}| \ge \eps |A|/(2K)$.
Moreover for any vertices $a_1, \dots, a_k \in A_*$ there are at least $(1-2k\eps)|\mathcal{V}|$ common neighbours $x\in \mathcal{V}$; i.e.~vertices  $x$ such that $(a_1,x), \dots, (a_k,x) \in \mathcal{E}$.
By putting $\eps = 1/(4k)$, we obtain the result.
$\hfill\Box$

\bigskip

Theorem \ref{t:BSzG_Schoen} enables one to achieve Balog-Wooley type decomposition results, dealing with the energies directly.


\begin{proof}(Sketch of proof of Proposition \ref{t:sp_E_E'})

Set $\E^\times (A) = |A|^3 /K$.
Using Theorem \ref{t:BSzG_Schoen} in multiplicative form
with $k=2$, we find a set $A_* \subseteq A$, $|A_*| \gg |A| /K$ and a set $P\subseteq A/A$, $|P| \ll K|A|$ such that for any $a,b \in A_*$ the following holds $|A\cap aP \cap bP| \gg |A|/K$.
Thus $\E^{+} (A_*)$
is bounded from above by
$$
    O\left( (|A| /K)^{-2}  |\{ a(p^{-1}_1 + p^{-1}_2) = b(p^{-1}_3 + p^{-1}_4) ~:~ p_1,p_2,p_3,p_4 \in P,\, a,b \in A \}|\right) \,.
$$
Using a consequence of the Szemer\'{e}di--Trotter theorem, see e.g. Lemma 2.5 from \cite{MR-NS}, we get
$$
    \E^{+} (A_*) \lesssim (|A| /K)^{-2} \cdot (\E^\times (A))^{1/2} |P|^3
        \ll
            K^5 |A| (\E^\times (A))^{1/2} \,.
$$
This completes the proof.
\end{proof}



\bigskip



\begin{thebibliography}{99}



\bibitem{AMRS}
{\sc E.~Aksoy Yazici, B.~Murphy, M.~Rudnev, I.D.~Shkredov, }
\emph{Growth estimates in positive characteristic via collisions, }
 Int Math Res Notices (2016), doi: 10.1093/imrn/rnw206.


\bibitem{BW}
{\sc A.~Balog, T.D.~Wooley, }
\emph{A low--energy decomposition theorem, } Quarterly Journal of Mathematics,  {\bf 68}:1 (2017), 207--226, doi: 10.1093/qmath/haw023.


\bibitem{Bg}
    {\sc J. Bourgain, }
    {\em More on the sum--product phenomenon in prime fields and its applications,}
    International Journal of Number Theory {\bf 1}:1 (2005), 1--32.

\bibitem{Bou09}
{\sc J.~Bourgain, }
\emph{Multilinear exponential sums in prime fields under optimal entropy condition on the sources, }
 Geom. Funct. Anal., {\bf 18} (2009), 1477--1502.



\bibitem{BouGar09}
{\sc J.~Bourgain, M.Z.~Garaev, }
\emph{On a variant of sum-product estimates and explicit exponential sum bounds in prime fields, }
Math. Proc. Cambridge Philos. Soc. {\bf 146}:1 (2009), 1--21.

\bibitem{BKT} {\sc J. Bourgain, N. Katz, T. Tao,}
{\em A sum-product estimate in finite fields, and applications,}
Geom. Funct. Anal. \textbf{14} (2004), 27--57.

\bibitem{CEIK}
{\sc J. Chapman, M.B. Erdo\u gan, D. Hart, A Iosevich, D. Koh, }
{\em Pinned distance sets, k-simplices, Wolff's exponent in finite fields and sum-product estimates,}
Mathematische Zeitschrift {\bf 271}:1-2  (2012), 63--93.


\bibitem{E} {\sc G. Elekes, } {\em On the number of sums and products,} Acta Arith. {\bf 81} (1997), 365--367.

\bibitem{ER} {\sc G. Elekes, I. Z.  Ruzsa,} {\em Few sums, many products,} Studia Sci. Math. Hungar. {\bf 40}:3 (2003), 301--308.

\bibitem{ES}
{\sc P.~Erd\H{o}s, E.~Szemer\'{e}di, }
\emph{On sums and products of integers, }
Studies in pure mathematics, 213--218, Birkh\"auser, Basel, 1983.

\bibitem{Fo} {\sc K. Ford, }
{\em The distribution of integers with a divisor in a given interval,} Annals of Math., {\bf 168} (2008), 367--433.


    \bibitem{Gow_m}
    {\sc W.T.~Gowers, }
    {\em A new proof of Szemer\'{e}di's theorem, }
    Geom. Funct. Anal. \textbf{11} (2001), 465--588.
    
    \bibitem{GT}  {\sc  B. Green, T. Tao,}
  {\em On Sets Defining Few Ordinary Lines, }
Discrete \& Computational Geometry
\textbf{50}:2 (2013), 409--468.

\bibitem{JR} {\sc T.G.F. Jones, O. Roche-Newton,} 
{\em Improved bounds on the set $A(A+1)$,} 
J. Combin. Theory Ser. A {\bf 120}:3 (2013), 515--526.

\bibitem{KR} {\sc S.V. Konyagin, M. Rudnev,} {\em On new sum-product-type estimates,} SIAM J. Discrete Math. {\bf 27}:2 (2013),  973--990.


\bibitem{KS1}
{\sc S.V.~Konyagin, I.D.~Shkredov, }
{\em On sum sets of sets, having small product sets, }
Proc. Steklov Inst. Math., {\bf 3}:290, (2015), 304--316.



\bibitem{KS2}
{\sc S.V.~Konyagin, I.D.~Shkredov, }
{\em New results on sum--products in $\mathbb R$, }
 Proc. Steklov Inst. Math., {\bf 294}:78, (2016), 87--98. 


\bibitem{MR-NS}
{\sc B.~Murphy, O.~Roche--Newton, I.D.~Shkredov, }
{\em Variations on the sum-product problem, }
SIDMA, {\bf 29}:1 (2015), 514--540.

\bibitem{Petridis}
{\sc{G.~Petridis, }}
{\em Products of Differences in Prime Order Finite Fields,}
arXiv:1602.02142 [math.CO] 5 Feb 2016.

\bibitem{RRS} {\sc O. Roche-Newton, M. Rudnev, I. D. Shkredov, }
{\em New sum-product type estimates over finite fields,}  Adv. Math. {\bf 293} (2016), 589--605.

\bibitem{misha}
{\sc M.~Rudnev,}
{\em On the number of incidences between planes and points in three dimensions,} Combinatorica (2017), 1--36, doi:10.1007/s00493-016-3329-6.

\bibitem{S_BSzG}
{\sc T.~Schoen, }
{\em New bounds in Balog--Szemer\'{e}di--Gowers theorem, }
Combinatorica, {\bf 34}:5 (2014), 1--7.



\bibitem{SS} {\sc T. Schoen, I.D. Shkredov,} {\em  Higher moments of convolutions, } J. of Number Theory, {\bf 133}
(2013), 1693--1737.

\bibitem{s_sumsets} {\sc I.D. Shkredov, }
{\em On sums of Szemer\'{e}di--Trotter sets, }
 Proc. Steklov Inst. Math., 289, (2015), 300--309.



\bibitem{Shkredov_R[A]}
{\sc I.D.~Shkredov, }
{\em Difference sets are not multiplicatively closed, }
Discrete Anal., {\bf 17} (2016), 1--21, doi: 10.19086/da.913.


\bibitem{s_E_k}
{\sc I.D.~Shkredov, }
{\em Some remarks on the Balog--Wooley decomposition theorem and quantities $D^+$, $D^\times$, }
Proc. Steklov Inst. Math., accepted,  arXiv:1605.00266v1 [math.CO] 1 May 2016.


\bibitem{S1}
{\sc J. Solymosi, }
{\em On the number of sums and products, }
 Bull. London Math. Soc. {\bf 37}: 4 (2005), 491--494.

\bibitem{Sol05}
{\sc J. Solymosi, }
{\em Bounding multiplicative energy by the sumset, }
Adv. Math. {\bf 222}: 2 (2009), 402--408.

\bibitem{SZ} {\sc S. Stevens, F. de Zeeuw, } {\em An improved point-line incidence bound,} arXiv: 1609.06284v2 [math.CO] 7 Oct 2016.

\bibitem{SzT} \sc{ E. Szemer\'edi, W. T. Trotter, } {\em Extremal problems in discrete geometry,} Combinatorica {\bf 3} (1983), 381--392.


\bibitem{TaoVubook}
{\sc T. Tao, V. Vu, }
{\em Additive Combinatorics, }
Cambridge University Press (2006).


\bibitem{To} C.~T\'oth, {\em The Szemer\'edi-Trotter theorem in the complex plane,} Combinatorica {\bf 3} (2015), no 1, 95--126.


\end{thebibliography}
\end{document}